\documentclass[a4paper,11pt,oneside]{article}
\usepackage[margin=1in]{geometry}
\usepackage{amsmath, amsthm, amssymb}
\usepackage{hyperref}
\usepackage{mathtools}
\usepackage{natbib}

\hypersetup{
	colorlinks=true,
  linkcolor=red,          %
  citecolor=blue,         %
  filecolor=magenta,      %
  urlcolor=cyan           %
}

\usepackage{ucs} 
\usepackage[utf8x]{inputenc}

\usepackage[T1]{fontenc}    %
\usepackage{hyperref}       %
\usepackage{url}            %
\usepackage{booktabs}       %
\usepackage{amsfonts}       %
\usepackage{nicefrac}       %
\usepackage{microtype}      %

\usepackage{mathtools}
\usepackage{amsthm, amssymb}

\newcommand{\bepf}{\begin{proof}}
\newcommand{\enpf}{\end{proof}}

\newtheorem{theorem}{Theorem}
\newtheorem{lemma}{Lemma}

\newtheorem{proposition}{Proposition}

\newcommand{\set}[1]{\left\{ #1 \right\}}

\newcommand{\pred}[1]{\boldsymbol{1}[#1]}

\newcommand{\R}{\mathbb{R}}

\newcommand{\N}{\mathbb{N}}

\newcommand{\inv}{^{-1}}

\DeclareMathOperator{\sgn}{sgn}

\newcommand{\abs}[1]{\left| #1 \right|}
\newcommand{\paren}[1]{\left( #1 \right)}
\newcommand{\sqprn}[1]{\left[ #1 \right]}
\newcommand{\nrm}[1]{\left\Vert #1 \right\Vert}

\newcommand{\vertiii}[1]{{\left\vert\kern-0.25ex\left\vert\kern-0.25ex\left\vert #1 
    \right\vert\kern-0.25ex\right\vert\kern-0.25ex\right\vert}}

\usepackage{fancyvrb} %

\newcommand{\mexp}{\mathbb{E}}

\newcommand{\PR}[2][]{\mathop{\mathbb{P}}_{#1}\left( #2 \right)}
\newcommand{\E}{\mathop{\mexp}}
\renewcommand{\P}{\mathbb{P}}

\newcommand{\eps}{\varepsilon}

\newcommand{\rad}{R_n}
\newcommand{\emprad}{\hat R_n}

\newcommand{\eqdef}{:=}

\usepackage{xcolor}

\newcommand{\decr}[1]{{#1}^{\downarrow}}

\newcommand{\beq}{\begin{eqnarray*}}
\newcommand{\eeq}{\end{eqnarray*}}
\newcommand{\beqn}{\begin{eqnarray}}
\newcommand{\eeqn}{\end{eqnarray}}

\newcommand{\QED}{\hfill\ensuremath{\square}}

\newcommand{\del}{\partial}

\newcommand{\diff}{\mathrm{d}}
\newcommand{\dd}[2]{\frac{\diff#1}{\diff#2}}

\newcommand{\ddel}[2]{\frac{\del#1}{\del#2}}

\usepackage{ifmtarg}
\usepackage{xifthen}%
\newcommand{\ent}[1][]{%
\ifthenelse{\isempty{#1}}{%
\mathrm{H}
}{
\mathrm{H}^{(#1)}
}}

\newcommand{\loch}[1][]{%
\ifthenelse{\isempty{#1}}{%
\mathrm{h}
}{
\mathrm{h}^{(#1)}
}}

\newcommand{\mathe}{\mathrm{e}}

\newcommand{\hide}[1]{}

\newcommand{\mx}{\vee}
\newcommand{\mn}{\wedge}

\renewcommand{\phi}{\varphi}

\newcommand{\ceil}[1]{\ensuremath{\left\lceil#1\right\rceil}}
\newcommand{\floor}[1]{\ensuremath{\left\lfloor#1\right\rfloor}}

\newcommand{\Bernu}{\operatorname{Bernoulli}}
\newcommand{\Binom}{\operatorname{Binomial}}
\newcommand{\mathd}{\mathrm{d}}

\newcommand{\lgc}{\mathsf{LGC}}

\newcommand{\pn}{\hat p_n}
\newcommand{\pp}{
[0,
{\textstyle\frac12}
]^\N_{\downarrow0}
}
\newcommand{\ppoh}{
[0,
{\textstyle\frac12}
]^\N
}

\newcommand{\subplus}{_{
\scalebox{.5}{
\!\!\!\!\!
$\boldsymbol{+}$}
}}
\newcommand{\pl}[1]{\sqprn{#1}\subplus}

\usepackage[ruled,vlined]{algorithm2e}

\usepackage{import}

\usepackage{amsthm}
\usepackage{mathtools}
\usepackage{bbm}

\title{
Local Glivenko-Cantelli\thanks{
Accepted for presentation at the Conference on Learning Theory (COLT) 2023
}
}

\author{%
  Doron Cohen  \\
  Department of Computer Science \\
	Ben-Gurion University of the Negev \\
  Beer-Sheva, Israel \\
  \texttt{doronv@post.bgu.ac.il} \\
  \and
  Aryeh Kontorovich \\
  Department of Computer Science \\
	Ben-Gurion University of the Negev \\
  Beer-Sheva, Israel \\
  \texttt{karyeh@cs.bgu.ac.il} \\
  }

\begin{document}

\maketitle

\begin{abstract}%
If $\mu$ is a distribution over the $d$-dimensional Boolean cube $\set{0,1}^d$, our goal is to estimate its mean $p\in[0,1]^d$ based on $n$ iid draws from $\mu$. Specifically, we consider the empirical mean estimator $\pn$ and study the expected
maximal deviation $\Delta_n=\E\max_{j\in[d]}|\pn(j)-p(j)|$. In the classical Universal Glivenko-Cantelli setting, one seeks distribution-free (i.e., independent of $\mu$) bounds on $\Delta_n$. This regime is well-understood: for all $\mu$, we have $\Delta_n\lesssim\sqrt{\log(d)/n}$ up to universal constants, and the bound is tight.

Our present work seeks to establish dimension-free (i.e., without an explicit dependence on $d$) estimates on $\Delta_n$, including those that hold for $d=\infty$. As such bounds must necessarily depend on $\mu$, we refer to this regime as {\em local} Glivenko-Cantelli (also known as $\mu$-GC), and are aware of very few previous bounds of this type --- which are 
either ``abstract'' or
quite sub-optimal. Already the special case of product measures $\mu$ is rather non-trivial. We give necessary and sufficient conditions on $\mu$ for $\Delta_n\to0$,
and calculate
sharp rates for this decay.
Along the way,
we discover a novel sub-gamma-type maximal inequality for shifted Bernoullis,
of independent interest.
\end{abstract}

\section{Introduction}
Estimating the mean of a random variable 
$X\in\R^d$
from
a sample of independent draws $X_i$
is among the most basic problems of statistics.
Much of the theory has focused on obtaining
efficient estimators $\hat m_n$
of the true mean $m$ and
analyzing the decay of
$\nrm{\hat m_n-m}_2$
as a function of sample size $n$, dimension $d$,
and various moment assumptions on $X$
\citep{Devroye16,LuMen19a,LuMen19b,Cherapanamjeri19,Cherapanamjeri20,LuMen21}.
In this work, we study this problem from a different angle. Inspired by \citet{cstheory-faster-low-ent},
we consider a distribution $\mu$
on $\set{0,1}^d$ with mean
$p\in[0,1]^d$.
We stress that $p$ is {\em not} a distribution:
the $p(j)$ do not generally sum to $1$ and
$\sum_{j=1}^d p(j)$
might well diverge
for $d=\infty$.
Given $n$ iid draws of $X_i\sim\mu$,
we denote by $\pn=n\inv\sum_{i=1}^n X_i$ the empirical mean. The central quantity of
interest in this paper is
the uniform absolute deviation
\beqn
\label{eq:kdef}
\Delta_n(\mu) := \E\nrm{\pn-p}_\infty
=
\E\max_{j\in[d]}|\pn(j)-p(j)|.
\eeqn

A few immediate remarks are in order.
First, the 
$\ell_\infty$
norm in \eqref{eq:kdef}
is in some sense the most interesting
of the $\ell_r$ norms;
indeed, for $r<\infty$,
$\Delta_n^{(r)}:=\E\nrm{\pn-p}_r^r$
decomposes into a sum of expectations
and the condition
$\Delta_n^{(r)}\to0$
reduces to one of convergence
of the appropriate
series.
Second, it is obvious that 
$\Delta_n(\mu)
\le1$
always,
and
$\Delta_n(\mu)\to0$ as $n\to\infty$
whenever $d<\infty$;
in fact, 
it is well-known 
(and proven below for completeness)
that
\beqn
\label{eq:distr-free}
\Delta_n(\mu) &
\lesssim
& 
\sqrt{\frac{\log(d+1)}{n}}
.
\eeqn
It is likewise clear that \eqref{eq:distr-free}
is not tight for all distributions:
if 
$X\sim\mu$ satisfies
$X(1)=X(2)=\ldots=X(d)$,
then $\Delta_n(\mu)$ is determined by
the single parameter $
p(1)=\E X(1)
$, and does not depend on $d$.
Our goal is to understand how $\Delta_n$
depends on the distribution
$\mu$ 
in a dimension-free fashion
---
i.e., without an explicit
dependence on $d$.

We 
broaden
our scope
to encompass the infinite-dimensional case
$d=
\infty$,
where $\mu$ is a distribution
on
$\set{0,1}^\N$
(supported on the usual $\sigma$-algebra
generated by the finite-dimensional cylinders),
and $\Delta_n(\mu)=\E\sup_{j\in\N}|\pn(j)-p(j)|$.
At the same time, 
for most of this paper
we narrow the scope 
to the case where $\mu$ is a product measure
on $\set{0,1}^\N$. Thus, the components $X(j)\sim\Bernu(p(j))$ are independent
and $\mu$ is entirely determined by the sequence $p\in[0,1]^\N$; our notation 
$\Delta_n(p)$ will 
indicate the product-measure setting.

For general $p\in[0,1]^\N$,
it no longer necessarily holds
that $
\Delta_n(p)\to0$. 
Indeed, taking
$p=(1/2,1/2,\ldots,)$
yields
$\Delta_n(p)=1/2$ for all $n$
(for the simple reason that if a random variable
has
a positive chance of attaining a certain value 
and
is given infinitely many opportunities to do so, it almost surely will).
A straightforward generalization
of this argument shows that 
$\lim_{j\to\infty}\min\set{p(j),1-p(j)}=0$
is a necessary condition for
$\Delta_n(p)\to0$.
Since 
$
|u-v|
=
|(1-u)-(1-v)|
$,
there is no loss of
generality in restricting our attention to $p\in\ppoh$
--- and further, in light of the preceding discussion, only to those $p$
for which $p(j)\to0$. Any such $p$ has a 
non-increasing permutation
$p^{\downarrow}$,
and since $\Delta_n(p)=\Delta_n(p^{\downarrow})$,
we will henceforth additionally assume, without loss of generality, that
$p=p^{\downarrow}$.

Having restricted our attention to
$\pp$
--- 
that is,
the collection of all $p\in\ppoh$ decreasing to $0$ ---
we distinguish a subset
$\lgc\subseteq 
\pp
$
via the criterion $\Delta_n(p)\to0$;
those
$p$ 
for which this holds
will be called
{\em local Glivenko-Cantelli}.\footnote{
Following \citet{vaart96},
we might also term this property $\mu$-Glivenko-Cantelli.
}
The central 
challenges posed
in this paper
are to give a precise characterization of
$\lgc$ as well as the rate at which
$\Delta_n(p)\to0$ as a function of $n$ and $p$.

\paragraph{Notation.}
Our logarithms will always be base $\mathe$ by default; other bases will be explicitly specified.
The natural numbers
are denoted by
$\N=\set{1, 2, 3, \dots}$
and
for $k\in\N$, we write $[k]=\set{
i\in\N:i\le k
}$.
For $n,d\in\N$ and a distribution $\mu$
over $\set{0,1}^d$,
we will always denote by
$p,\pn\in[0,1]^d$ the true and
empirical means of $\mu$, respectively,
as defined immediately preceding \eqref{eq:kdef}.
The $j$th coordinate of a vector $v\in\R^d$
is denoted by $v(j)$.
The expected maximal deviation
$\Delta_n(\mu)$ is defined in \eqref{eq:kdef},
and
the notation
$\pp$, $\lgc$
from the preceding paragraph
will be used throughout.
We say that $\mu$ is a {\em product}
distribution on $\set{0,1}^d$
if it can be expressed as a 
tensor
product
of $d$ distributions:
$\mu=\mu_1\otimes\mu_2\otimes\ldots\otimes\mu_d$;
equivalently, for $X\sim\mu$ we have that
the random variables $\set{X(j):j\in[d]}$
are mutually independent.

These definitions continue to hold
when $d=\infty$,
though
some care must be taken in defining the $\sigma$-algebra on $\set{0,1}^\N$;
see, e.g., \citet{Kallenberg02}.
Alternatively, one 
considers
a sequence of
cubes $\set{0,1}^d$, $d\in\N$,
along with a sequence of
distributions $\mu_d$ on $\set{0,1}^d$,
such that for each $d'<d''$, $\mu_{d'}$
coincides with the marginal distribution
of $\mu_{d''}$
on the first $d'$ coordinates.
Then the Ionescu Tulcea
extension
Theorem 
\citep[Theorem 5.17]{Kallenberg02}
guarantees that the $\mu_d$'s can be ``stitched together''
consistently
into a probability measure $\mu$ on $\set{0,1}^\N$
(equipped with the $\sigma$-algebra
generated by the finite-dimensional cylinders). Further,
Lebesgue's Monotone Convergence Theorem implies that
\beqn
\label{eq:ionescu}
\E_{X\sim\mu}\sup_{i\in\N}
|\pn(j)-p(j)|
=
\sup_{d\in\N}
\E_{X\sim\mu_d}\max_{i\in[d]}
|\pn(j)-p(j)|
.
\eeqn
Anyone wishing to side-step the measure theory may take
the right-hand-side of \eqref{eq:ionescu}
as the {\em definition}
of the left-hand side.

For $f,g:\N\to(0,\infty)$, we write
$f\lesssim g$
if
$\limsup_{n\to\infty}f(n)/g(n)<\infty$.
Likewise, $f\gtrsim g\iff g\lesssim f$
and $f\asymp g$ if both 
$f\lesssim g$
and
$f\gtrsim g$
hold.
The floor and ceiling functions, $\floor{t}$,
$\ceil{t}$,
map $t\in\R$
to its closest integers below and above, respectively;
also,
$s\mx t:=\max\set{s,t}$,
$s\mn t:=\min\set{s,t}$,
and $\pl{s}:=0\mx s$.
Unspecified constants
such as $c,c'$
may change value
from line to line.

\section{Main results}
\label{sec:main-res}

We begin with a characterization of $\lgc$.
For each
$p\in\pp$, we define the two key quantities,
\beqn
\label{eq:Sdef}
S(p) &:=& \sup_{j\in\N}p(j)\log (j+1),
\\
\label{eq:Tdef}
T(p) &:=& \sup_{j\in\N}\frac{\log (j+1)}{\log(1/p(j))}.
\eeqn
For any (not necessarily product)
probability measure $\mu$ on $\set{0,1}^\N$,
recall that $p=
{\displaystyle\E_{X\sim\mu}[X]}$
and define $\tilde p(j)=\min\set{p(j),1-p(j)}$, $j\in\N$.
Whenever $\tilde p(j)\to 0$, we have
that $\decr{\tilde p}$ is well-defined,
as are $S(\mu):=S(\decr{\tilde p})$
and $T(\mu):=T(\decr{\tilde p})$;
otherwise,
$S(\mu),T(\mu):=\infty$.
\begin{theorem}
[Characterization of $\lgc$]
\label{thm:lgc}
Any
$p\in\pp$ 
satisfies $\Delta_n(p)\to0$
if and only if
$T(p)<\infty$.
Additionally,
if $\mu$
is {\em any} probability measure on
$\set{0,1}^\N$
with
$T(\mu)<\infty$,
then $\Delta_n(\mu)\to0$.
\end{theorem}

It is immediate from the definitions 
of $S
$ and $T
$ 
that $S(p)\le T(p)$ and hence $S(p)<\infty$ 
whenever
$p\in\lgc$. For these $p$, the asymptotic decay of
$\Delta_n(p)\to0$ is determined by $S(p)$:
\begin{theorem}
[Coarse asymptotics of $\Delta_n$]
\label{thm:asympS}
For every
product 
probability
measure $\mu$
on $\set{0,1}^\N$
such that
$p
=\E_{X\sim\mu}[X]
\in\lgc$,
we have
\beq
c\sqrt{S(p)} \le 
\liminf_{n\to\infty}
\sqrt n
\Delta_n(p)
\le
\limsup_{n\to\infty}
\sqrt n
\Delta_n(p)
\le
c'\sqrt{S(p)}
,
\eeq
where $c,c'$ are absolute constants.
Additionally,
if $\mu$
is {\em any} probability measure on
$\set{0,1}^\N$
and
$T(\mu)
<\infty
$,
then
$
\limsup_{n\to\infty}
\sqrt n
\Delta_n(\mu)
\le
c'\sqrt{S(\mu)}
$.

\end{theorem}
This result may be informally summarized as
$\sqrt n
\Delta_n(p)
\asymp \sqrt{S(p)}$,
for product measures.
In addition to the asymptotics,
we obtain the finite-sample upper bound
\begin{theorem}
\label{thm:Delta-sub-gamma}
For any
probability measure
$\mu$
on
$\set{0,1}^\N$
with
$T(\mu)<\infty$,
\beq
\Delta_n(\mu)
&\le&
c\paren{
\sqrt{\frac{S(\mu)}{n}}
+\frac{T(\mu)\log n}{n}
},
\qquad n\ge \mathe^3,
\eeq
where $c>0$ is an absolute constant.
\end{theorem}
We conjecture that the $\log n$ factor
multiplying $T(\mu)$
can be removed;
this would imply the 
improved
asymptotic
rate
$
\Delta_n(\mu)
\lesssim
\sqrt{S(\mu)/n}
+T(\mu)/n
$.
No further improvement is possible,
as evident from the
$n\Delta_n(p)
\gtrsim
T
$
asymptotic lower bound:
\begin{theorem}
[Fine asymptotics of $\Delta_n$]
\label{thm:T/n-lb}
For every
product 
probability
measure $\mu$
on $\set{0,1}^\N$
with
$p
=\E_{X\sim\mu}[X]
$
such that $T(p)<\infty$,
we have
\beq
\liminf_{n\to\infty}
n
\Delta_n(p)
\ge
cT(p)
,
\eeq
where $c$ is an absolute constant.
\end{theorem}

In contrast with the
finite-sample upper bound
of Theorem~\ref{thm:Delta-sub-gamma},
the lower bounds in Theorems
\ref{thm:asympS}
and
\ref{thm:T/n-lb}
are only asymptotic, and necessarily so. This is because even for a single binomial $Y\sim B(n,p)$, the behavior of $\mathbb{E}|Y-
np
|$ is roughly $np(1-p)$
for $p\notin[1/n,1-1/n]$ and $\approx \sqrt{np(1-p)}$ elsewhere 
\citep[Theorem 1]{Berend2013}.
Thus, there can be no lower bound of the form $\Delta_n\ge c\sqrt{S/n}$
or $\Delta_n\ge c'{T/n}$ that holds for all $n$ and all $p\in[0,1/2]^d$.

Instrumental in proving 
Theorem~\ref{thm:Delta-sub-gamma}
is a novel
sub-gamma 
inequality
for shifted Bernoulli distributions, of independent interest:
\begin{lemma}[Sub-gamma inequality for the shifted Bernoulli]
\label{lem:bernstein}
For all 
$
0
<
p
\le
s
\le
\mathe^{-3} 
$ and 
$0\le t< \log(1/p)/\log(1/s)$,
\beq
\E_{X\sim\Bernu(p)}[\exp(t(X-s))]
\le
\exp\paren{
\frac{pt^2}{
2[1-t\log(1/s)/\log(1/p)]
}
}
.
\eeq
\end{lemma}
This is a refinement of Bernstein's inequality: for $s=p$,
the latter is recovered up to constants. However, for $p\ll s$, the former is significantly sharper.

Finally, we provide
a fully empirical
upper bound
on $\Delta_n(\mu)$:
\begin{theorem}
\label{thm:ub-emp}
For any
probability measure
$\mu$
on
$\set{0,1}^\N$,
let $\hat\mu_n$
be its empirical realization
induced by
an iid sample $X_i\sim\mu$;
thus, $\hat\mu_n(x)=n\inv\sum_{i=1}^n\pred{X_i=x}$
and $\pn=
{\displaystyle
\E_{X\sim\hat\mu_n}[X]
}$.
Let $\Tilde{X}_i\sim\mu$ 
be
another iid sample of size $n$ independent 
of $X_i$ and 
define
$\tilde{p}_n \in [0,1]^\N$ 
by
$\tilde{p}_n(j) = (1-a(j))\pn(j) + a(j)(1-\pn(j))$,
where $ a(j) = \pred{ \sum_{i=1}^n\Tilde{X}_i(j) > \frac{n}{2} } $ .
Then
\beq
\Delta_n(\mu)
&\le&
\frac{16}{\sqrt n}
\sqrt{
S(\decr{\Tilde{p}_n})
}
+ 
\sqrt{\frac{8 \log(1/\delta)}{n}}
\eeq
holds with probability at least $1-\delta$.
(We interpret $S(\decr{u})$
as $\infty$ when $\decr{u}$
does not exist.)
\end{theorem}

An attractive feature of this bound is that it is stated entirely in terms of quantities easily computed from the
observed sample --- unlike, say, the upper bound
in Theorem~\ref{thm:Delta-sub-gamma}, which
is stated in terms of the unknown $p(j)$.
When $S(\mu),T(\mu)$ are small, we
expect, from Theorem~\ref{thm:Delta-sub-gamma}
and
\eqref{eq:mcd}, 
for $\nrm{\pn-p}_\infty$ to be small
--- and hence,
$
S(\decr{\Tilde{p}_n})
\approx
S(\hat\mu_n)
\approx
S(\mu)
$.
In the ``unlucky'' case where
Theorem~\ref{thm:ub-emp}
fails to give a good empirical bound,
we can chalk it up, with some confidence,
to the badness of $\mu$.
As will become apparent from the proof,
the claim holds for {\em any}
choice of $a\in[0,1]^\N$
--- whether deterministic
or a function of the $\tilde X_i$.
Of course, for imprudent choices of $a$,
the quantity $S(\decr{\Tilde{p}_n})$
will fail to be small for ``well-behaved''
distributions $\mu$. Tying the typical behavior
of $S(\decr{\Tilde{p}_n})$
(for our choice of $a$)
to the well-behavedness of $\mu$
is a subject of future work.

\paragraph{Remark.}
Our upper bounds hold for all
probability measures
$\mu$ on $\set{0,1}^\N$,
while the lower bounds
in Theorem~\ref{thm:lgc}
($T(p)=\infty\implies p\notin\lgc$),
Theorem~\ref{thm:asympS}
($
\liminf\sqrt{n}\Delta_n(p)
\gtrsim
\sqrt{S(p)}
$),
and
Theorem~\ref{thm:T/n-lb}
($
\liminf {n}\Delta_n(p)
\gtrsim
{T(p)}
$)
make critical use of the product structure
of $\mu$.
The
upper bounds 
in Theorems
\ref{thm:asympS},
\ref{thm:Delta-sub-gamma},
\ref{thm:ub-emp}
are quite loose
when the coordinates $X(j)$ are strongly correlated.
Understanding the 
behavior of $\Delta_n(\mu)$
for non-product $\mu$
is
an active current research direction of ours.
When the pairwise correlations are negative
--- i.e., when
$\E[X(j)X(k)]\le p(j)p(k)$
for all $j\neq k$, ---
all of the results stated in this
paper for product measures
continue to hold, with
only a small change of multiplicative constants \citep{kontorovich2023decoupling,CohenK23b}.

\section{
Discussion
and
comparisons with known 
results
}
\label{sec:known-res}

\paragraph{Discussion.}
We argue that the bounds
in Theorems~\ref{thm:asympS} and~\ref{thm:Delta-sub-gamma}
are at least mildly surprising.
Indeed, it is known that
for $
X\sim\Bernu(p)
$,
its 
optimal
sub-Gaussian
variance proxy
(i.e., the smallest $\sigma^2$
such that $\E\mathe^{t(X-p)}\le 
\exp(t^2 \sigma^2/2)
$)
for all $t\in\R$
is given by
\beq
\sigma^2(p)
&=&
\frac{1-2p}{2\log(1/p-1)}
\eeq
\citep{DBLP:conf/uai/KearnsS98,ECP2359,Buldygin2013}
--- and hence, $
X
\sim n\inv\Binom(n,p)
$ is $\sigma^2(p)/n$-sub-Gaussian.
For $p\ll1$, we have $
\sigma^2(p)\asymp1/\log(1/p)$.
Thus, drawing intuition from the majorizing measure theorem
\citep[Theorem 6.24]{van2014probability},
one might expect
that $\Delta_n(p)\lesssim\sqrt{T(p)/n}$
captures the correct behavior for the case of product measures.
While 
this estimate indeed holds
(as an immediate consequence of
Lemma~\ref{lemma:max-subgaussian}),
it is
far from tight,
as evident from
Theorem~\ref{thm:asympS}.
Instead, $\Delta_n(p)$ exhibits both a sub-Gaussian
decay regime, with rate $\sqrt{S(p)/n}$
and a sub-exponential regime, with rate $\lesssim T\log(n)/n$;
this type of decay 
(without the $\log n$
factor)
is sometimes referred to as sub-gamma
\citep{blm13}.

Intuitively, the crucial difference between the %
(normalized) Binomial
and the Gaussian cases is that the former is absolutely bounded,
while the latter is not. Not only is 
$X\sim
n\inv
\Binom(n,p)
$
bounded in $[0,1]$,
but for $p\le s\ll1$
the shifted variable
$X-s$ will typically attain very small values\footnote{
Our motivation for
considering
$s>p$
will become apparent
in the sequel.
}.
Bernstein's classic inequality,
up to constants,
upper-bounds
$
\log
\E\mathe^{t(X-p)}
$
by
$
pt^2/
(1-t)
$
and holds for any
$X$
with range in $[0,1]$,
$\E X=p$,
and variance $\lesssim p$.
The refined estimate
in
Lemma~\ref{lem:bernstein}
shows that 
for $Y\sim\Bernu(p)$,
the ``effective
upper range'' of $Y-s$
is, 
in a useful sense,
something like $
\log(1/p)/\log(1/s)
$
---
which is much more
delicate than bounding
by the constant $1$
and
is precisely what
allows us to obtain
the sub-gamma tail.

\paragraph{Comparisons.}
The $\mu$-Glivenko-Cantelli ($\mu$-GC) property has
a few classical abstract characterizations.
\citet[Theorem 3.3]{vapnik98}
shows that
a concept class $F$
is $\mu$-GC if and only if
the $\mu$-expectation of the log-number
of behaviors achieved by $F$
on an $n$-size sample is sublinear in $n$.
Another classical
characterization of $\mu$-GC
is in terms of the empirical Rademacher
complexity
\citep[Theorem 4.10, Proposition 4.12]{wainwright2019high};
see the proof of Theorem~\ref{thm:ub-emp}.
These abstract characterizations should, in principle, imply our
Theorem~\ref{thm:lgc} --- though it is not at all obvious
how to derive the $T(p)<\infty$ characterization for our special case.
A somewhat related problem of testing product distributions of
Bernoulli vectors was recently studied 
by \citet{MR4510329}.

\citet{cstheory-faster-low-ent}
(effectively\footnote{
To be precise, the question
was regarding
the tail behavior of
$\nrm{\pn-p}_\infty$
rather than
its expectation
$\Delta_n(\mu)$.
})
asked whether
$\Delta_n(\mu)$
can be bounded
in terms of the entropy of $\mu$.
For product measures on $\set{0,1}^\N$
parametrized by $p\in\pp$,
the entropy is given by
\beq
H(p) = -\sum_{j\in\N}p(j)\log p(j)-(1-p(j))\log(1-p(j))
\eeq
and $H(p)<\infty$ is a much stronger condition
than $T(p)<\infty$;\footnote{
If $H(p)<\infty$ then 
certainly $\sum_{j\in\N}p(j)<\infty$,
which implies $T(p)<\infty$
via Lemma~\ref{lem:B-CK}.
On the other hand,
for $p(j)=1/j$, we have $T<\infty$
while $H=\infty$.
We note in passing that
for 
small $x$
and
$p=(x,0,0,\ldots)$,
we have $T(p)=1/\log(1/x)\gg x\log(1/x)
\approx
H(p)$,
so $T(p)\le a H(p)^b$
does not, in general, hold
for any constants $a,b>0$.
}
thus, in light of Theorem~\ref{thm:lgc},
$H(p)$ is not, in general, the correct measure
for controlling the decay of $\Delta_n(\mu)$.
(Of course, for non-product measures $\mu$,
the entropy $H(\mu)$ takes coordinate correlations
into account and can be significantly smaller than
$T(\mu)$ and even than $S(\mu)$.)

Already 
in \citet{cstheory-faster-low-ent},
it was observed that Hoeffding's inequality together with the union bound imply
\beqn
\label{eq:ub-worst-case}
\P(\nrm{\pn-p}_\infty\ge\eps)
\le 2d\mathe^{-2n\eps^2},
\qquad
\eps>0,n\in\N.
\eeqn
Hence,
a sample of size $n\ge\frac{\log(2d/\delta)}{\eps^2}$
suffices to achieve
$\P(\nrm{\pn-p}_\infty\ge\eps)\le\delta$.
This easily
implies \eqref{eq:distr-free},
which 
is worst-case tight, as witnessed by the uniform distribution.
It was also 
noted
therein 
that McDiarmid's inequality 
implies
\beqn
\label{eq:mcd}
\P(\nrm{\pn-p}_\infty\ge
\Delta_n(\mu)
+
\eps)
\le \mathe^{-2n\eps^2},
\qquad
\eps>0,n\in\N,
\eeqn
which reduces the problem to one of
estimating $\Delta_n(\mu)
$. 
An elementary Borel-Cantelli argument
shows that $\Delta_n(\mu)\to0$
if and only if $\nrm{\pn-p}_\infty\to0$
almost surely.
Moreover, standard information-theoretic
techniques can be used to show
that the distribution-free upper bound
in \eqref{eq:distr-free}
is worst-case tight not just for the
empirical mean $\pn$,
but also for 
{\em any other} estimator $\tilde p_n$,
see Proposition~\ref{cor:vc}.
This continues to hold
even if we restrict our attention
only to product measures $\mu$,
as the proof thereof shows.

Additional estimates on $\Delta_n(\mu)$
suggested in 
\citet{cstheory-faster-low-ent}
include
\beq
\Delta_n(\mu)
&\le&
\sqrt{\frac1n\sum_{j\in\N}p(j)(1-p(j))}
\eeq
and
\beq
\Delta_n(\mu)&\le&\sqrt{\frac12H(\mu)}.
\eeq
The former is considerably inferior to the bound in
Theorem~\ref{thm:Delta-sub-gamma},
while the latter does not decrease as $n\to\infty$.

\section{
Proofs and proof sketches
}
\label{sec:proofs}

\subsection{Proof sketch for Lemma~\ref{lem:bernstein} (Sub-gamma inequality for the shifted Bernoulli)}

Once we parametrize $t=a\log(1/p)/\log(1/s)$ with $0 \leq a < 1$, proving the inequality is equivalent to showing that the following expression, $F(a)$, 
is 
non-positive:
\beq
F(a):=
\log \left(\left(\frac{1}{p}\right)^{\frac{a (1-s)}{\log \left(\frac{1}{s}\right)}-1}+(1-p) \left(\frac{1}{p}\right)^{-\frac{a s}{\log \left(\frac{1}{s}\right)}}\right)-\frac{a^2 p \log ^2\left(\frac{1}{p}\right)}{2 (1-a) \log ^2\left(\frac{1}{s}\right)}
\le 0.
\eeq
This inequality can be demonstrated using elementary calculus techniques. The full proof is available in Appendix \ref{subsec:bernstein-proof}.

\subsection{
Proof of 
Theorem~\ref{thm:Delta-sub-gamma} 
}
We argue that there is no loss of generality in assuming
$p \in \pp$:
replacing $X(j)$
by $1-X(j)$
does not affect $\Delta_n(\mu)$
--- even in the non-product case.
Decompose:
\beq
    \Delta_n(\mu)
&=&
    \E\sup_{j\in \N}
{
  \pl{\pn(j)-p(j)}
  \mx
  \pl{p(j)-\pn(j)}
  }
\\
& \leq&
    \E\sup_{j\in \N}
    \pl{\pn(j)-p(j)}
    +
    \E\sup_{j\in \N}
    \pl{p(j)-\pn(j)}
    .
\eeq
To bound the lower tail
$\E\sup_{j\in \N} \pl{p(j)-\pn(j)}$
we first invoke
the Chernoff-type bound of
\citet[Theorem~2, (ii)]{okamoto1959some}
(see also \citet[Exercise 2.1.2]{blm13})
to obtain
\beq
    \PR{p(j)-\pn(j) \geq t}
    & \leq&
    \exp \paren{- \frac{n t^2}{2 p(j) (1-p(j))}}
,
\qquad
j\in\N, t\ge0.
\eeq
Applying
Lemma~\ref{lemma:max-subgaussian} with $Y_j=p(j)-\pn(j)$ and $ \sigma_j^2 = \frac{p(j)}{n}$
yields
\beqn
\E\sup_{j\in \N} \pl{p(j)-\pn(j)}
&\leq&
4 \sup_{j\in\N}\sqrt{\frac{p(j)}{n} \log (j+1)}
\leq
4\sqrt{\frac{S(\mu)}{n}}.
\label{eq:lower-tail-bound}
\eeqn
It remains to estimate the upper tail
$
\E\sup_{j\in \N}
\pl{\pn(j) - p(j)}
$.
To this end,
we decompose
\beqn
\label{eq:upper-tail-decomposition}
    \E\sup_{j\in \N}
    \pl{\pn(j) - p(j)}
    & \leq&
    \E\sup_{\substack{ j\in \N \\ p(j) \ge \frac{1}{n}} }
    \pl{\pn(j) - p(j)}
    +
    \E\sup_{\substack{ j\in \N \\ p(j) < \frac{1}{n}} }
    \pl{\pn(j) - p(j)}
\eeqn
and focus on the first term 
$ 
\E\sup_{
p(j) \ge \frac{1}{n}
}
\pl{\pn(j) - p(j)} 
$. 
For each $ j $, we apply Bernstein's inequality \cite[Proposition~2.14]{wainwright2019high}
to obtain
\beq
        \PR{\pn(j) - p(j) \geq \eps}
        &\leq&
        \exp \paren{- \frac{\eps^2 }{2 \paren{\frac{p(j)(1-p(j))}{n} + \frac{\eps}{3n} } } }
        ,
\qquad
\eps > 0.
\eeq
Now invoke Lemma~\ref{lem:max-subgamma} with $Y_i = \pn(i) - p(i),\ v_i=\frac{p(i)(1-p(i))}{n},\ a_i= \frac{1}{3n},\ I=\set{i\in \N : p(i) \ge \frac{1}{n}}$
to yield
\beqn
    \E{\sup_{p(j) \ge \frac{1}{n}} \pl{\pn(j) - p(j)}  } 
    &\leq&
    12 \sup_{p(j) \ge \frac{1}{n}}
    \sqrt{ \frac{p(j)(1-p(j))}{n} \log (j+1) }
    +
     \frac{16}{3n}\sup_{p(j) \ge \frac{1}{n}}  \log (j+1)
    \nonumber
    \\
    & \leq&
    12 \sqrt{\frac{S(\mu)}{n}}
    +
    \frac{16}{3n} \sup_{p(j) \ge \frac{1}{n}}  \log (j+1)
    \nonumber
    \\
    & \leq&
    12 \sqrt{\frac{S(\mu)}{n}}
    +
    \frac{16 T(\mu) \log n}{3n},
    \label{eq:upper-tail-big-p}
\eeqn
where 
\eqref{eq:upper-tail-big-p}
holds because
$ \sup_{p(i) \geq 1/n}\frac{\log (i+1)}{\log(1/p(i))} \leq T$ implies $ \log (i+1) \leq T \log n$.
It remains to estimate the second term
in the right-hand side of
\eqref{eq:upper-tail-decomposition}, which we 
decompose as
\beqn
\E\sup_{\substack{ j\in \N \\ p(j) < \frac{1}{n}} }
    \pl{\pn(j) - p(j)}
    &    \leq&
    \E\sup_{\substack{ j\in \N \\ p(j) < \frac{1}{n}} }
    \pl{\pn(j) - \frac{1}{n}}
    +
    \E\sup_{\substack{ j\in \N \\ p(j) < \frac{1}{n}} }
    \pl{\frac{1}{n} - p(j)}
    \nonumber
    \\
    &    \leq&
    \E\sup_{\substack{ j\in \N \\ p(j) < \frac{1}{n}} }
    \pl{\pn(j) - \frac{1}{n}}
    +
    \frac{1}{n}.
    \label{eq:upper-small-p}
\eeqn
To upper-bound
the first term
$ \E\sup_{
p(j) < \frac{1}{n}
}
\pl{\pn(j) - \frac{1}{n}} $,
we will use Lemmas
\ref{lem:bernstein} 
and
\ref{lem:max-subgamma}.
Recall that
$ p_n(j)-s \eqdef \frac{1}{n}\sum_{i=1}^n X_i(j) - s$ 
where $X_i(j) \sim \Bernu\paren{p(j)}$ and $ X_1(j), X_2(j), \dots , X_n(j)$ are mutually independent. Let $\mathe^{-3} > s \ge p(j)$ and let $0\le t< \log(1/p(j))/\log(1/s).$ 
Then
\beq
\E{\exp\paren{t(p_n(j)-s)}}
& =&
        \prod_{i=1}^n
        \E{\exp\paren{\frac{t}{n}(X_i(j)-s)}}
\\ & \leq&
        \prod_{i=1}^n
        \exp\paren{
        \frac{p(j)t^2}{
        2n^2 [1-t\log(1/s)/\ (n\log(1/p(j)))]
        }
        }
\\ & =&
        \exp\paren{
        \frac{(p(j)/n)t^2}{
        2 [1-t\log(1/s)/(n\log(1/p(j)))]
        }
        }.
\eeq
Put $ s = {1}/{n}$ and apply Lemma~\ref{lem:max-subgamma} with 
$Y_i = \pn(i) - \frac{1}{n}$, $v_i=\frac{p(i)}{n}$,
$a_i= \log(1/s)/(n\log(1/p(j)))$,
and
$I=\set{i\in \N 
:
p(i) < \frac{1}{n}}$, which yields,
together with
\eqref{eq:upper-small-p},
\begin{align}
    \E\sup_{\substack{ j\in \N \\ p(j) < \frac{1}{n}} }
    \pl{\pn(j) - p(j)}
    & \leq
    12 \sup_{p(j) < \frac{1}{n}}
    \sqrt{ \frac{p(j)}{n} \log (i+1) }
    +
     \frac{16\log(n)}{n}\sup_{p(j) < \frac{1}{n}}  \frac{\log (j+1)}{\log(1/p(j))}
    +
    \frac{1}{n}
    \nonumber
    \\ & \leq
    12 \sqrt{\frac{S(\mu)}{n}}
    +
     \frac{16T(\mu) \log n}{n}
    +
    \frac{1}{n}.
    \label{eq:upper-tail-small-p}
\end{align}
Summing up
\eqref{eq:lower-tail-bound}, \eqref{eq:upper-tail-big-p}, and \eqref{eq:upper-tail-small-p},
we conclude that
\beq
    \Delta_n(\mu) \leq 28 \paren{\sqrt{\frac{S(\mu)}{n}} + \frac{T(\mu) \log n}{n}} + \frac{1}{n},
\eeq
which proves Theorem~\ref{thm:Delta-sub-gamma} for $T(\mu)$
sufficiently large
--- say,
$T(\mu)\geq \frac{1}{2}$.
We now assume $T(\mu) < \frac{1}{2}$ and
decompose as in
\eqref{eq:upper-tail-decomposition}
but at
a different the splitting point:
\beq
    \E\sup_{j\in \N}
    \pl{\pn(j) - p(j)}
    & \leq&
    \E\sup_{ j \leq n^\frac{T(\mu)}{1-T(\mu)} }
    \pl{\pn(j) - p(j)}
    +
    \E\sup_{ j > n^\frac{T(\mu)}{1-T(\mu)} }
\pl{\pn(j) - p(j)}.
\eeq
In order to bound the first term, 
$\E\sup_{ j \leq n^\frac{T(\mu)}{1-T(\mu)} }
\pl{\pn(j) - p(j)} $, 
we follow the same steps 
that we did to bound
$ \E\sup_{
p(j) \ge \frac{1}{n}
}
\pl{\pn(j) - p(j)} $ and get, instead of \eqref{eq:upper-tail-big-p},
\beqn
\E\sup_{ j \leq n^\frac{T(\mu)}{1-T(\mu)} }
\pl{\pn(j) - p(j)}
    & \leq&
    12 \sqrt{\frac{S(\mu)}{n}}
    +
    \frac{16}{3} \frac{T(\mu) \log n}{n ({1-T(\mu)})}
    \nonumber
\\ & \leq&
    12 \sqrt{\frac{S(\mu)}{n}}
    +
    11 \frac{T(\mu) \log n}{n}.
    \label{eq:upper-tail-small-T-big-p}
\eeqn
For the second term, $\E\sup_{ j > n^\frac{T(\mu)}{1-T(\mu)} }
\pl{\pn(j) - p(j)}$, we note that for $j > n^\frac{T(\mu)}{1-T(\mu)}$, we have
    \begin{align*}
        p(j) \leq
        \frac{1}{(j+1)^{1/T(\mu)}}
        \leq
        \frac{1}{\paren{n^\frac{T(\mu)}{1-T(\mu)}+1}^{1/T(\mu)}}
        <\frac{1}{n}.
    \end{align*}
It is well-known
\citep[Theorem~1]{MR1144242,Berend2013} that for
$p(j) \leq \frac{1}{n} $,
we have
$\E \abs{\pn(j) - p(j)} \leq 2 p(j)$. 
Consequently,
\beqn
    \E\sup_{ j > n^\frac{T(\mu)}{1-T(\mu)} }
\pl{\pn(j) - p(j)}
    & \leq&
    \E\sup_{ j > n^\frac{T(\mu)}{1-T(\mu)} }
    \abs{\pn(j) - p(j)}
    \nonumber
    \\ & \leq&
    \sum_{j > n^\frac{T(\mu)}{1-T(\mu)}} 2 p(j)
    \nonumber
    \\ & \leq&
    \sum_{j > n^\frac{T(\mu)}{1-T(\mu)}}  \frac{2}{(j+1)^{1/T(\mu)}}
    \nonumber
    \\ & \leq&
    \int_{{}_
    {\!\!\!\!\!n^\frac{T(\mu)}{1-T(\mu)}}}
    ^\infty  \frac{2}{u^{1/T(\mu)}} \mathd u
    \nonumber
    \\ & =&
    \frac{2 T(\mu) \left(n^{\frac{T(\mu)}{T(\mu)-1}}\right)^{\frac{1}{T(\mu)}-1}}{1-T(\mu)}
    \nonumber
    \\ & =&
    \frac{2 T(\mu)}{n(1-T(\mu))}
    \nonumber
    \\ & \leq&
    \frac{4 T(\mu)}{n}.
    \label{eq:upper-tail-small-T-small-p}
\eeqn
Summing up \eqref{eq:lower-tail-bound}, \eqref{eq:upper-tail-small-T-big-p}, and \eqref{eq:upper-tail-small-T-small-p}, we conclude
\beq
    \Delta_n(\mu) &\leq& 16 \paren{ \sqrt{\frac{S(\mu)}{n}} +  \frac{ T(\mu)}{n}}
\eeq
for $T(\mu) < \frac{1}{2}$.
This completes the proof.
\qed

\subsection{Proof of Theorem~\ref{thm:lgc}}

Theorem~\ref{thm:Delta-sub-gamma} immediately implies that $T(p)<\infty\implies p\in\lgc$.
Indeed, since $p\le1/\log(1/p)$, we have that $S(p)\le T(p)$
and hence $\Delta_n(\mu)\lesssim\sqrt{T(\mu)/n}+T(\mu)\log(n)/n$,
which, for finite $T(\mu)$, obviously decays to $0$ as $n\to\infty$.

The other direction,
$
p\in\lgc\implies
T(p)<\infty
$,
is an immediate consequence of
Theorem~\ref{thm:T/n-lb}.
However, we find it
instructive to give a more
elementary and intuitive (though less quantitative) proof,
based on
an observation of \citet{dani2022}.
For any $p\in[0,1]^\N$, let us say that 
it satisfies condition (B)
if 
\beq
\text{(B)}
\qquad \qquad 
\inf_{k\in\N}\sum_{j\in\N}p(j)^k<\infty
\eeq
(an appeal to Lebesgue's monotone convergence theorem shows that
the above expression is either $0$ or $\infty$).

\begin{lemma}[\citet{dani2022}]
\label{lem:berend}
If $p\in[0,1/2]^\N$ does not satisfy 
\textup{(B)}
then $\Delta_n(p)\ge c>0$
for some absolute constant $c$.
\end{lemma}

\begin{lemma}
\label{lem:BCK}
\label{lem:B-CK}
For $p\in\pp$, the conditions 
\textup{(B)}
and $T(p)<\infty$ are equivalent.
\end{lemma}
The proof for Lemmas, \ref{lem:B-CK} and \ref{lem:berend} can be found in Appendices \ref{proof:B-CK} and \ref{proof:berend}, respectively.\\
\noindent {\bf Remark.}
Observe that (B) is permutation-invariant while $T(p)<\infty$ assumes the decreasing ordering,
hence the two conditions are only equivalent on $\pp$.

Combining Lemmas~\ref{lem:berend} and~\ref{lem:BCK} immediately implies that $T(p)=\infty\implies p\notin\lgc$.

\subsection{Proof of Theorem~\ref{thm:asympS}}
The upper estimate $
\sqrt{n}\Delta_n(\mu)\lesssim\sqrt{S(\mu)}
$
is immediate from Theorem~\ref{thm:Delta-sub-gamma}.
Thus, it only remains to prove the lower estimate
\beqn
\label{eq:lower-asymp}
\liminf_{n\to \infty} 
\sqrt{n}\Delta_n(p)
\ge
c\sqrt{S(p)}
\eeqn
for some absolute constant $c>0$.
This result is 
actually
subsumed in our proof
Theorem~\ref{thm:T/n-lb}, but we include the somewhat
simpler proof below, which also has the advantage of yielding explicit constants.

We will use
the following
``Reverse Chernoff bound'' 
due to \citet[Lemma~4]{KleinYoung15}:
\begin{lemma}
\label{lem:rev-chern}
Suppose that $X\sim\Binom(n,p)$,
and $0<\eps,p\le1/2$
satisfy $\eps^2pn\ge3$.
Then
\beq
\P(X\ge (1+\eps)pn)
&\ge&
\exp\big({-9\eps^2 pn}\big).
\eeq
\end{lemma}
\noindent
We will also use 
\citet[Theorem~1]{Berend2013}:
\begin{lemma}
\label{lem:BK}
Suppose that 
$n\ge2$,
$p\in[1/n,1-1/n]$,
and
$X\sim\Binom(n,p)$.
Then
\beq
\E|X-np|
&\ge&
\sqrt{\frac12\E(X-np)^2}
=
\sqrt{\frac12 np(1-p)}\\
&\ge& 
\frac12\sqrt{np},\qquad p\le1/2
.
\eeq
\end{lemma}

\noindent
Continuing with the proof
of %
\eqref{eq:lower-asymp},
we put
$\eps = \sqrt{\frac{5 \log (j+1)}{n p(j)}} $.
Then
Lemma~\ref{lem:rev-chern}
implies that
for each
$j \in \N$
verifying
$
\frac{20 \log(j+1)}{n} \leq p(j),
$
\beq
\PR{\abs{\pn(j) - p(j)} \geq \sqrt{\frac{5 p(j) \log (j+1)}{n}}}
&\geq&
\frac{1}{\paren{j+1}^{45}}.
\eeq
Since 
the $\pn(j)$,
$j\in\N$,
are assumed to be independent,
for all natural $k \leq l$
\begin{align}
\label{eq:prod-young}
\P
\paren{\max_{j\in [k,l]} \abs{\pn(j) - p(j)}
\leq
\sqrt{\frac{5 p(l)\log(k+1)}{n}
}}
\leq
\paren{1-\frac{1}{(k+1)^{45}}}^{l-k+1},
\end{align}
whenever $p(l) \geq 20\log (k+1)/n$.
For $ k\in \N$, define
\beq
J(k)
\eqdef
\set{
2^{45^{k-1}}-1,
2^{45^{k-1}},
2^{45^{k-1}}+1,
2^{45^{k-1}}+2,
\ldots
,
2^{45^k}}.
\eeq
A repeated application of \eqref{eq:prod-young} yields,
for 
$p(2^{45^{k}}) \geq 20\log (2^{45^{k-1}})/n$,
\begin{align*}
\P\paren{\max_{j\in J(k)} \abs{\pn(j) - p(j)}
\leq
\frac{1}{\sqrt{45}}\sqrt{\frac{p(2^{45^{k}})\log {(2^{45^{k}})}}{n}
}}
&=
\P\paren{\max_{j\in J(k)} \abs{\pn(j) - p(j)}
\leq
\sqrt{\frac{p(2^{45^{k}})\log {(2^{45^{k-1}})}}{n}
}}
\\
&\leq
\paren{1-\frac{1}{2^{45^k}}}^{2^{45^k}-2^{45^{k-1}}}
\\
&\leq
\exp\paren{-2^{-45^k}}
^{2^{45^k}-2^{45^{k-1}}}
\\
&=
\exp\paren{2^{-44 \cdot 45^{k-1}}-1}
\\&<
\frac{1}{\mathe}.
\end{align*}
If follows that
\beqn
    \Delta_n (p)
    &\ge&
    \max_{k \in \N} \E \max_{j\in J(k)} \abs{\pn(j) - p(j)}
    \nonumber
    \\
    &\geq&
    \max_{\substack{k \in \N \\ 
    n p(2^{45^{k}}) \geq 20\log \paren{2^{45^{k-1}}}
    }}
    \paren{1-\mathe^{-1}}\frac{1}{\sqrt{45}}\sqrt{\frac{p(2^{45^k})\log {(2^{45^k})}}{n}}
    \nonumber
    \\
    &\geq&
    \max_{\substack{k \in \N \\ 
    n p(2^{45^{k}}) \geq 20\log \paren{2^{45^{k-1}}}
    }}
    \frac{1}{90}
    \sqrt{\frac{p(2^{45^k})\log {(2^{45^{k+1}})}}{n}}
    \nonumber
    \\
    &\geq&
    \max_{\substack{k \in \N \\ 
    n p(2^{45^{k}}) \geq 20\log \paren{2^{45^{k-1}}}
    }}
    \max_{j \in J(k+1)}
    \frac{1}{90}
    \sqrt{\frac{p(2^{45^k})\log {(j+1)}}{n}}
    \nonumber
    \\
    &\geq&
    \max_{\substack{k \in \N \\ 
    n p(2^{45^{k}}) \geq 20\log \paren{2^{45^{k-1}}}
    }}
    \max_{j \in J(k+1)}
    \frac{1}{90}
    \sqrt{\frac{p(j)\log {(j+1)}}{n}}
    \nonumber
    \\
    &\geq&
    \max_{\substack{k \in \N}}
    \max_{\substack{ j \in J(k+1)
    \\ 
    n p(j) \geq 20\log \paren{j+1}
    }}
    \frac{1}{90}
    \sqrt{\frac{p(j)\log {(j+1)}}{n}}
    \nonumber
     \\
    &=&
    \frac{1}{90\sqrt{n}}
    \max_{\substack{ j > 2^{45}
    \\ 
    n p(j) \geq 20\log \paren{j+1}
    }}
    \sqrt{p(j)\log {(j+1)}}.
    \label{eq:lower-by-young}
\eeqn
Additionally,
using $
\E\sup_{j\in\N} \abs{\pn(j) - p(j)}
\ge
\sup_{j\in\N}\E \abs{\pn(j) - p(j)}
$
and 
Lemma~\ref{lem:BK}
for $ n \geq 2$, we have
\beqn
    \Delta_n (p)
    &\ge&
    \max_{j \leq 2^{45}}
    \E \abs{\pn(j) - p(j)}
    \nonumber
    \\
    & \ge&
    \frac{1}{2\sqrt{n}}
    \max_{\substack{j \leq 2^{45} \\ p(j) \geq \frac{1}{n}  } } \sqrt{p(j)}
    \nonumber
    \\
    & \ge&
    \frac{1}{90\sqrt{n}}
    \max_{\substack{j \leq 2^{45} \\ 
    n p(j) \geq 20\log \paren{j+1}
    } } \sqrt{p(j) \log(j+1)}.
    \label{eq:lower-by-berend}
\eeqn
Finally, we combine \eqref{eq:lower-by-young} and \eqref{eq:lower-by-berend} 
to obtain
\beq
    \sqrt{n} \Delta_n (p)
    &=&
    \frac{1}{2}\sqrt{n} \Delta_n (p) + \frac{1}{2}\sqrt{n} \Delta_n (p)
    \\
    &\ge&
    \frac{1}{180}
    \max_{\substack{ j > 2^{45}
    \\ n p(j) \geq 20\log \paren{j+1}}}
    \sqrt{p(j)\log {(j+1)}}
    +
    \frac{1}{180}
    \max_{\substack{j \leq 2^{45} \\ n p(j) \geq 20\log \paren{j+1} } } \sqrt{p(j) \log(j+1)}
    \\
    &\ge&
    \frac{1}{180}
    \max_{\substack{n p(j) \geq 20\log \paren{j+1} } } \sqrt{p(j) \log(j+1)}
\eeq
for $ n \geq 2$. Taking limits yields \eqref{eq:lower-asymp}, with $c=1/180$.

\qed

\subsection{Proof sketch for Theorem~\ref{thm:T/n-lb}}

The proof for this result is similar to the proof for the lower bound in Theorem~\ref{thm:asympS}. However, instead of using the anti-concentration bound from Lemma~\ref{lem:rev-chern}, we use a different anti-concentration bound stated in Lemma~\ref{lem:zz}, from \citet[Theorem~9]{zz2020}. This Lemma is a bit cumbersome to work with, so we simplify it further through Lemma~\ref{lem:kl-upper-bound}. The rest of the proof resembles the proof for the lower bound in Theorem~\ref{thm:asympS}. The full proof can be found in Appendix~\ref{proof:T/n-lb}.

\subsection{Proof sketch for Theorem~\ref{thm:ub-emp}}

The goal is to upper bound 
$$\Delta_n(\mu) = \E\sup_{f \in F} \abs{n\inv\sum_{i=1}^n f(X_i)-\E f(X_i)},$$ where $F$ is the class of functions 
$$f_j=(1-a(j))x(j) + a(j)(1 - x(j))$$ over $\Omega = \set{0,1}^\N$ defined conditionally on $ \Tilde{X}_i$. The proof uses a symmetrization argument along with McDiarmid's inequality to bound $\Delta_n(\mu)$ around the empirical Rademacher average of $F$ with high probability. Finally, the moment-generating function of each term in the empirical Rademacher average is bounded using Hoeffding's lemma, then Lemma~\ref{lemma:max-subgaussian} is used to complete the proof. The full proof can be found in Appendix~\ref{proof:ub-emp}.

\paragraph{Acknowledgements.}
We thank 
Daniel Berend
and
Steve Hanneke
for the numerous helpful discussions.
This research was partially supported by
the Israel Science Foundation
(grant No. 1602/19), an Amazon Research Award,
and the Ben-Gurion University Data Science Research Center.

\bibliographystyle{plainnat}
\bibliography{../refs}

\appendix

\section{Deferred proof}

\subsection{Proof of Lemma~\ref{lem:bernstein} (Sub-gamma inequality for the shifted Bernoulli)}
\label{subsec:bernstein-proof}

Let us parameterize
$t=
a\log(1/p)/\log(1/s)
$
for $0\le a<1$; this captures the exact range
of the allowed values of $t$.
Proving our inequality amounts to showing that
\beq
F(a):=
\log \left(\left(\frac{1}{p}\right)^{\frac{a (1-s)}{\log \left(\frac{1}{s}\right)}-1}+(1-p) \left(\frac{1}{p}\right)^{-\frac{a s}{\log \left(\frac{1}{s}\right)}}\right)-\frac{a^2 p \log ^2\left(\frac{1}{p}\right)}{2 (1-a) \log ^2\left(\frac{1}{s}\right)}
\le 0.
\eeq
We claim that 
$F(0)=0$ and $F'(a)\le 0$.
The former is immediate, and to show the latter,
we
compute the derivative:
\beq
F'(a)
=
\frac{\log \left(\frac{1}{p}\right) \left(2 \log \left(\frac{1}{s}\right) \left(\frac{p^{1-\frac{a}{\log \left(\frac{1}{s}\right)}}}{p^{1-\frac{a}{\log \left(\frac{1}{s}\right)}}-p+1}-s\right)+\frac{(a-2) a p \log \left(\frac{1}{p}\right)}{(a-1)^2}\right)}{2 \log ^2\left(\frac{1}{s}\right)}.
\eeq
The factors
$
\log \left(\frac{1}{p}\right)
$
and
$
\log ^2\left(\frac{1}{s}\right)
$
are positive;
additionally,
$
{p^{1-\frac{a}{\log \left(\frac{1}{s}\right)}}-p+1}
\geq
1
$;
hence,
it remains to show that
\beq
G
=
2 \log \left(\frac{1}{s}\right) \left(p^{1-\frac{a}{\log \left(\frac{1}{s}\right)}}-s\right)+\frac{(a-2) a p \log \left(\frac{1}{p}\right)}{(a-1)^2}
<0,
\eeq
since $\sgn(G) \geq \sgn(F'(a))$.
Let us parameterize
by
$u=1/\log(1/s)<1$;
then
\beq
G
=
\frac{2 \left(p^{1-a u}-\mathe^{-1/u}\right)}{u}+\frac{(a-2) a p \log \left(\frac{1}{p}\right)}{(a-1)^2} .
\eeq

Now
\beq
\ddel{G}{u}
=
-\frac{2 \mathe^{-1/u} p^{-a u} \left(-(u-1) p^{a u}+a p \mathe^{1/u} u^2 \log (p)+p \mathe^{1/u} u\right)}{u^3},
\eeq
and since
both
$u^3$
and
$2 \mathe^{-1/u} p^{-a u}$
are non-negative,
the sign of
$
\ddel{G}{u}
$
is determined by
\beq
H=
(u-1) p^{a u}-a p \mathe^{1/u} u^2 \log (p)-p \mathe^{1/u} u.
\eeq
Further,
\beq
\ddel{H}{a}
=
u \log (p) \left((u-1) p^{a u}-p \mathe^{1/u} u\right)
\ge0
\eeq
since $u<1$
and $\log p<0$.
Thus, $H$ is maximized at
$a=1$, with a value of
\beq
H_1
=
(u-1) p^u
-p \mathe^{1/u} u^2 \log (p)
-p \mathe^{1/u} u
.
\eeq
We now show that
$H_1\le0$.
We have $u\le1/3$ by the assumption $s \leq \mathe^{-3}$.
Then
\beq
H_1 &\le&
-\frac23p^{1/3}-\mathe^3pu
+ 
u^2 \mathe^{1/u} p \log\frac1p
=:\tilde H_1.
\eeq
Now
\beq
\ddel{\tilde H_1}{u}
&=&
-p \left(e^{1/u} (1-2 u) \log\frac1p+\mathe^3\right)
\le0,
\eeq
whence
$\tilde H_1$
is decreasing in $u$.
Thus, to show that $\tilde H_1<0$, it suffices to evaluate
$\tilde H_1$
at the smallest allowed value 
of $u=1/\log(1/p)$.
The latter evaluates to
\beq
-\frac{2}{3} p^{\frac{1}{\log \left(\frac{1}{p}\right)}}
-\frac{\mathe^3 p}{\log \left(\frac{1}{p}\right)}
+\frac{1}{\log\left(\frac{1}{p}\right)}
,
\eeq
which is easily seen to be 
$\le0$ for $p\in[0,1/2]$.
Indeed,
parametrizing
$v=1/\log(1/p)$
and differentiating 
with respect to
$v$,
we
get
\beq
J(v) &:=& -\mathe^{3-\frac{1}{v}} v+v-\frac{2}{3 e}
\\
J'(v)&=&1-\frac{\mathe^{3-\frac{1}{v}} (v+1)}{v}
\\
J''(v)&=& -\frac{\mathe^{3-\frac{1}{v}}}{v^3} < 0
.
\eeq
Solving for $J'(v)=0$
yields $v^*=
\frac{1}{-W_{-1}\left(-\frac{1}{\mathe^4}\right)-1}
\approx
0.21
$,
where $W_{-1}$ is the Lambert $W$ function
at the $-1$ branch.
Since $J(v^*)
\approx
-0.071
<0$, 
we conclude that
$
\tilde H_1(
1/\log(1/p)
)
\le0
$.

It follows that $H\le 0$,
whence
$\ddel{G}{u}\le0$.
Since $G$ is decreasing in $u$, it is also decreasing in $s$
(because
$u(s)=1/\log(1/s)$ is monotonically increasing),
and this it suffices to evaluate $G$ at $s=p$,
which yields
\beq
\left(\frac{(a-2) a p}{(a-1)^2}
+
2 \left(\mathe^a-1\right) (1-p) p\right) \log \left(\frac{1}{p}\right).
\eeq
Now
\beq
\dd{}{p}
\sqprn{
\frac{(a-2) a p}{(a-1)^2}-2 \left(\mathe^a-1\right) (1-p) p
}
=
\frac{(a-2) a}{(a-1)^2}-2 (\mathe^a-
1)(1-2p)<0,
\eeq
so it suffices to consider $G$ at $s=p=0$,
where it is $0$.

\qed

\subsection{Proof of Lemma~\ref{lem:berend}} %
\label{proof:berend}
The negation of (B) means that $\sum_{j\in\N}p(j)^k=\infty$ for all $k\in\N$. Thus,
\beq
\E\sup_{j\in\N}|\pn(j)-p(j)|
& \geq&
    \frac{1}{2} \PR{\sup_{j\in\N}\pn(j)-p(j) \geq \frac{1}{2}}
\\ & \geq&
    \frac{1}{2} \paren{1- \mathe^{-1}} \paren{1\wedge \sum_{j \in \N} \PR{\pn(j)-p(j) \geq \frac{1}{2}}}
\quad \text{(\citet[Problem~5.1a]{van2014probability})}
\\ & \geq&
\frac{1}{2} \paren{1- \mathe^{-1}} \paren{1\wedge \sum_{j \in \N} \PR{\pn(j) = 1}}
\\ & =&
\frac{1}{2} \paren{1- \mathe^{-1}} \paren{1\wedge \sum_{j \in \N} p(j)^n}
\\ & =&
    \frac{1}{2} \paren{1- \mathe^{-1}}.
\eeq
\QED

\subsection{Proof of Lemma~\ref{lem:B-CK}} %
\label{proof:B-CK}
The direction
$
T(p)<\infty
\implies
\text(B)
$
is 
obvious.
Indeed,
$T(p) < \infty$ means that there is a $T>0$ such that %
$p(j) \leq 1/(j+1)^{1/T}$
for all $j\in\N$.
Then $\sum_{j=1}^\infty p(j)^k \leq \sum_{j=1}^\infty 1/(j+1)^{k/T} <\infty$ for $k>T$.

To show that 
$
\mathrm{(B)}
\implies 
T(p)<\infty 
$, assume $T(p)=\infty$
and
define $R(j) \eqdef \frac{\log(j+1)}{\log(1/p(j))}$, %
$ j \in \N$;
thus,
$ T(p) = \sup_{j \in \N} R(j)$.
We make two observations: 
(i) $\limsup_{j\in\N}R(j)\ge T
$
implies
$p(j)\ge 1/(j+1)^{1/T}$ for infinitely many $j$
and
(ii) $R(j)\ge T
$
implies
$R(\lceil j/2 \rceil)\ge T-2$
via
\beq
\log(\lceil j/2 \rceil + 1)/\log(1/p(\lceil j/2 \rceil)) 
    &\geq&
\log((j+2)/2)/\log(1/p(\lceil j/2 \rceil)) 
\\&\geq&
(\log(j+2)-1)/\log(1/p(j))
\\&\geq&
(\log(j+1)-1)/\log(1/p(j))    \geq R(j) - 2,
\eeq
where the monotonicity of $p(j)$ was used.
Assume, to get a contradiction, that
$
\sum_{j\in\N} p(i)^k<\infty
$
for some $k\in\N$
and choose $T=2k+2$.
By (i) and (ii) above,
$R(j)\ge T$
and $R(\lceil j/2 \rceil)\ge T-2$
holds for infinitely many $j\in\N$.
Invoking monotonicity again,
\beq
\sum_{i=1}^\infty p(i)^k
&\ge&
\sum_{i=\lceil j/2 \rceil }^j p(i)^k
\\&\ge&
\sum_{i=\lceil j/2 \rceil }^j 
\sqprn{1/(i+1)^{1/(T-2)}}
^k
\\&\ge&
\frac{j}{4}
\cdot
\frac{1}{(j+1)^{k/(T-2)}}
\\&\ge&
\frac{j}{4\sqrt{j+1}} .
\eeq
The latter 
holds
for infinitely many $j$,
whence the left-hand side is unbounded --- a contradiction.
\qed

\subsection{Proof of Theorem~\ref{thm:T/n-lb}}
\label{proof:T/n-lb}

For $p,q\in(0,1)$,
we define the
Kullback-Leibler
and $\chi^2$
divergences, respectively,
between
the distributions
$\Bernu(p)$
and
$\Bernu(q)$:
\beq
D(p~\Vert~ q) &=&
p\log\frac{p}{q}
+
(1-p)\log\frac{1-p}{1-q}
,\\
\chi^2(p~\Vert~ q)
&=&
\frac{(p-q)^2}{q}
+
\frac{(p-q)^2}{1-q}
.
\eeq

\begin{lemma}
\label{lem:kl-upper-bound}
For
$p \in (0,1/2] $ 
and
$\eps \in [0,1-p]$,
we have
\beqn
\label{eq:kl-inequality}
D(\eps + p~\Vert~ p)
& \leq &
2 \min \set{
\eps\log(1/p)
,
\frac{\eps^2}{p}
}
.
\eeqn

\begin{proof}
\citet[Theorem~5]{gibbs02} 
states that
$
D(p~\Vert~ q)
\le
\log(1+\chi^2(p~\Vert~ q))
$.
Thus,
\beq
\label{KLupperGibbs}
D(p + \eps ~\Vert~ p)
& \leq&
\log \paren{1 + 
\chi^2(p + \eps~\Vert~ p)}
\\
& =&
\log \paren{1 + \frac{\eps^2}{1-p}+\frac{\eps^2}{p}}
\\
& \leq& \frac{\eps^2}{1-p}+\frac{\eps^2}{p}
\\
& = & \frac{\eps^2}{p(1-p)},
\\
& \leq& \frac{2\eps^2}{p}.
\eeq
The second inequality,
$
 2\eps\log(1/p) - D(p + \eps ~\Vert~ p) \geq 0$,
holds for endpoints
$ x = 0$ and $x=1-p$ and
\beq
\dd{^2}{x^2}
\sqprn{
2\eps\log(1/p) - D(p + \eps ~\Vert~ p)
}
=
-\frac{1}{(1-p-x) (p+x)}
\le
0
\eeq
for $0<x<1-p$, such that $ 2\eps\log(1/p) - D(p + \eps ~\Vert~ p) $ is concave and hence non-negative in the claimed range.
\end{proof}
\end{lemma}
We will also make use of a result
of
\citet[Theorem~9]{zz2020}:
\begin{lemma}
\label{lem:zz}
For any $\beta > 1$ 
there exist constants 
$ c_\beta, C_\beta > 0$,
depending only on $\beta$,
such that
whenever
$0 \leq \eps \leq \frac{1-p}{\beta}$
and
$\eps + p \geq \frac{1}{n}$,
we have
\beq
\PR{\pn - p \geq \eps}
&
\geq
&
c_\beta \exp \paren{-C_\beta n D(\eps + p~\Vert~ p)}
.
\eeq
\end{lemma}

\begin{proof}[Proof of Theorem~\ref{thm:T/n-lb}]
For $ k,l \in \N$, define
\beq
\eps(k,l) &\eqdef& \max \set{ \frac{\log(k+1)}{n\log (1/p(l))},\sqrt{ \frac{p(l)\log(k+1)}{n} }},
\\
\eps(k) &\eqdef& \eps(k,k)
.
\eeq
Invoking
Lemma~\ref{lem:kl-upper-bound}
and
Lemma~\ref{lem:zz}
with $\beta = 2$,
we have
that for each 
$j \in \N$
verifying
$\frac{1}{n} \leq \eps(j) \leq \frac{1}{4}$,
\beq
\PR{\pn(j) - p(j) \geq  \eps(j) }
&\geq&
c_2 \exp \paren{-C_2 n D(p(j) + \eps(j) ~\Vert~ p(j)) }
\\
&\geq&
c_2 \exp \paren{-2C_2 n \min \paren{
\eps\log(1/p(j))
,
\frac{\eps(j)^2}{p(j)}
} }
\\
&\geq&
c_2 \exp \paren{-2 C_2  \log(j+1) }
\\
&=&
\frac{c_2}{(j+1)^{2 C_2}}.
\eeq

Since the $\pn(j)$ are assumed to be independent,
for all natural $k \leq l$
we have
\beqn
\label{eq:lower-zz-single-j}
\PR{\max_{j\in [k,l]} \pn(j) - p(j)
\leq
\eps(k,l)
}
& \leq &
\PR{\bigwedge_{j\in [k,l]} \pn(j) - p(j)
\leq
\eps(k,j)
}
\nonumber
\\
& \leq &
\paren{1 - \frac{c_2}{(k+1)^{2C_2}} }^{l-k+1},
\eeqn
whenever
$ \eps(k) \leq \frac{1}{4} $ and $\eps(k,l) \geq \frac{1}{n}$.
For $ k \in \N$, define
\[
J(k) \eqdef \set{2^{(2C_2)^{k-1}}-1, 2^{(2C_2)^{k-1}}, 2^{(2C_2)^{k-1}} + 1 , 2^{(2C_2)^{k-1}} + 2 , \ldots , 2^{(2C_2)^{k}}}.
\]
where we assume without loss of generality $2C_2 \in \N$.
For $k \in \N$, 
define
\beq
\eta(k) \eqdef \eps\paren{2^{(2C_2)^{k-1}}-1,2^{(2C_2)^{k}}}
.
\eeq
A repeated application of \eqref{eq:lower-zz-single-j} yields, for 
$ k \in \N $ such that $\eta(k) \geq \frac1n$ and $\eps\paren{2^{(2C_2)^{k-1}}-1} \leq \frac14$, 
\beq
\PR{\max_{j \in J(k)} \pn(j) - p(j) 
\leq
\eta(k)
}
& \leq &
\paren{1 - \frac{c_2}{(2^{(2C_2)^{k-1}})^{2C_2}} }^{2^{(2C_2)^{k}}-2^{(2C_2)^{k-1}}}
\\
& = &
\paren{1 - \frac{c_2}{2^{(2C_2)^{k}}} }^{2^{(2C_2)^{k}}-2^{(2C_2)^{k-1}}}
\\
& \leq &
\paren{ \exp \paren{ - \frac{c_2}{2^{(2C_2)^{k}}} } }^{2^{(2C_2)^{k}}-2^{(2C_2)^{k-1}}}
\\
& = &
\exp \paren{ - \frac{c_2}{2^{(2C_2)^{k}} } \paren{2^{(2C_2)^{k}}-2^{(2C_2)^{k-1}}} } 
\\
& = &
\exp \paren{  -c_2 + c_2 2^{(2C_2)^{k-1} - 2^{(2C_2)^{k} } } }
\\
& < &
\mathe^{-c_2}.
\eeq
It follows that
\beqn
\label{eq:lower-bound-big-j}
\Delta_n(p)
&\geq &
\sup_{k \in \N} \E \max_{j \in J(k)} \pn(j) - p(j) 
\nonumber
\\
&\geq &
\paren{1 - \mathe^{-c_2}} \sup_{k \in A} \eta(k)
\nonumber
\\
& =  &
\paren{1 - \mathe^{-c_2}} \sup_{k \in A}
\max \set{ \frac{\log(2^{(2C_2)^{k-1}})}{n\log (1/p(2^{(2C_2)^{k}}))},\sqrt{ \frac{p(2^{(2C_2)^{k}})\log(2^{(2C_2)^{k-1}})}{n} }}
\nonumber
\\
& \geq &
\frac{1 - \mathe^{-c_2}}{4C_2^2}
\sup_{k \in A}
\max_{j \in J(k+1)}
\max \set{ \frac{\log(j+1)}{n\log (1/p(2^{(2C_2)^{k}}))},\sqrt{ \frac{p(2^{(2C_2)^{k}})\log(j+1)}{n} }}
\nonumber
\\
& \geq &
\frac{1 - \mathe^{-c_2}}{4C_2^2}
\sup_{k \in A}
\max_{j \in J(k+1)}
\max \set{ \frac{\log(j+1)}{n\log (1/p(j)))},\sqrt{ \frac{p(j)\log(j+1)}{n} }}
\nonumber
\\
& \geq &
\frac{1 - \mathe^{-c_2}}{4C_2^2}
\sup_{ \substack{j > 2^{2C_2}\\ 
{4C_2^2}
\leq n\eps(j) 
\leq \frac{n}{16C_2^2} } }
\eps(j),
\eeqn
where 
$A \eqdef \set{k \in \N 
:
\eps\paren{2^{(2C_2)^{k-1}}-1} \leq \frac{1}{4} \text{ and } \eta(k) \geq \frac1n }$.

It remains
to handle the initial segment $J(1)$:
\beqn
\label{eq:lower-bound-small-j}
\Delta_n(p)
&\geq &
\max_{j \in J(1)} \E  \pn(j) - p(j)
\nonumber
\\
& \geq &
\max_{\substack{j \in J(1) \\ \frac1n \leq \eps(j) \leq \frac{1}{4} } }
\frac{c_2}{(j+1)^{2 C_2}} \eps(j)
\nonumber
\\
& \geq &
\frac{c_2}{\paren{2^{(2C_2)}+1}^{2 C_2}}
\max_{\substack{j \in J(1) \\ \frac1n \leq \eps(j) \leq \frac{1}{4} } }
\eps(j).
\eeqn
Combining \eqref{eq:lower-bound-big-j} and \eqref{eq:lower-bound-small-j} yields
\beq
\Delta_n(p)
&\geq &
\frac12 \Delta_n(p) + \frac12 \Delta_n(p)
\\
&\geq &
\frac{1 - \mathe^{-c_2}}{8C_2^2}
\sup_{ \substack{j > 2^{2C_2}\\ \frac{4C_2^2}{n} \leq \eps(j) \leq \frac{1}{16C_2^2} } }
\eps(j)
+
\frac{c_2}{2 \paren{2^{(2C_2)}+1}^{2 C_2}}
\max_{\substack{j \in J(1) \\ \frac1n \leq \eps(j) \leq \frac{1}{4} } }
\eps(j)
\\
&\geq &
\min \set{\frac{c_2}{2 \paren{2^{(2C_2)}+1}^{2 C_2}}, \frac{1 - \mathe^{-c_2}}{8C_2^2}}
\sup_{ \substack{j \in \N \\ 
{4C_2^2}
\leq n\eps(j) 
\leq \frac{n}{16C_2^2} 
} }
\eps(j).
\eeq
Since
for every fixed $j\in\N$,
the condition
$
{4C_2^2}
\leq n\eps(j) 
\leq \frac{n}{16C_2^2} 
$
is eventually satisfied
as $n\to\infty$,
the claim is proved.
\end{proof}

\subsection{Proof of Theorem~\ref{thm:ub-emp}}
\label{proof:ub-emp}

We start by, conditionally on $ \Tilde{X}_i$, defining the class $F=\set{f_j:j\in \N}$ over $\Omega=\set{0,1}^\N$,
where $f_j(x)=(1-a(j))x(j) + a(j)(1 - x(j))$ and observing that 
\beq
\Delta_n(\mu) &=&
\E\sup_{f \in F} \abs{n\inv\sum_{i=1}^n f(X_i)-\E f(X_i)}.
\eeq
Combining with
\citet[Proposition~4.11]{wainwright2019high} --- a standard symmetrization argument --- we get
\begin{align}
\label{eq:symmetrization}
\Delta_n(\mu)
\le
2 \rad(F),
\end{align}
where
\beq
\emprad(F;X)
\eqdef
\E_{\eps} \sup_{f \in F}\abs{n\inv\sum_{i=1}^n \eps_i f(X_i)}
\\
\rad(F)
\eqdef
\E_X\emprad(F;X)
=
\E_{\eps,X} \sup_{f \in F}\abs{n\inv\sum_{i=1}^n \eps_i f(X_i)}
\eeq
are the 
(empirical and expected, respectively)
\emph{Rademacher complexities}; the $\eps_i, i\in [n]$ are independent Rademacher random variables
defined by $\P(\eps_i=1)=\P(\eps_i=-1)=1/2$.
Since
$\emprad(F;X)$
has $ 2/n $-bounded differences as a function of $ X_1,X_2,...X_n$, we can invoke McDiarmid's inequality
\citep[Theorem 6.2]{blm13}
to obtain, for all $ \delta \geq 0$,
\begin{align}
    \PR{ \Delta_n(\mu)
    \le
    2 \emprad(F;X) + 2\sqrt{\frac{2}{n} \log \frac{1}{\delta} } }
    &
    \geq
    \PR{ \rad(F)
    \le
    2 \emprad(F;X) + 2\sqrt{ \frac{2}{n} \log \frac{1}{\delta} } }
    &
    \text{[by \eqref{eq:symmetrization}]}
    \nonumber
    \\
    & \geq
    1-\delta.
    \label{eq:mcdiarmid-application}
\end{align}
We now turn to bounding $\emprad(F;X)$, by bounding the 
moment-generating function (conditional on $X_i$) of each 
$ n\inv\sum_{i=1}^n \eps_i f_j(X_i)$ for 
$
j\in \N$ via Hoeffding's lemma.
Evidently, for $\lambda \ge 0$, we have
\beq
    \E_{\eps} \exp \paren{\lambda \paren{n\inv\sum_{i=1}^n \eps_i f_j(X_i)}}
    & = &
    \prod_{i=1}^n \E_{\eps_i} \exp \paren{\lambda \paren{n\inv \eps_i f_j(X_i)}}
    \\
    &\leq&
    \prod_{i=1}^n \exp \paren{\lambda^2 \frac{(1-a(j))X_i(j) + a(j)(1 - X_i(j))}{2 n^2}}
    \\
    & = &
    \exp \paren{\lambda^2 \sum_{i=1}^n \frac{1-a(j))X_i(j) + a(j)(1 - X_i(j))}{2 n^2}}
    \\
    & = &
    \exp \paren{\lambda^2 \frac{1}{2n} \Tilde{p}_n(j) }.
\eeq
By the Cram\'er-Chernoff method,
\beq
    \PR[\eps]{n\inv\sum_{i=1}^n \eps_i f_j(X_i) \geq t}
    & \leq& \inf_{\lambda \geq 0} \frac{\exp \paren{\lambda^2 \frac{1}{2n} \Tilde{p}_n(j) }}{\mathe^{t \lambda} }
    =
    \exp \paren{-\frac{n t^2}{2 \Tilde{p}_n(j) }}.
\eeq
Conditional on $X_1,X_2,...,X_n$, we apply Lemma~\ref{lemma:max-subgaussian} with $Y_j = n\inv\sum_{i=1}^n \eps_i f_j(X_i)$, $\sigma^2_j = \frac{\Tilde{p}_n(j)}{n}$ and arrive with a bound for $ \emprad(F;X) $;
\begin{align*}
\emprad(F;X)
& \eqdef
\E_{\eps} \sup_{f \in F}\abs{n\inv\sum_{i=1}^n \eps_i f(X_i)}
\\
& \leq
\E_{\eps} \sup_{f \in F}\pl{n\inv\sum_{i=1}^n \eps_i f(X_i)} + \E_{\eps} \sup_{f \in F}\pl{n\inv\sum_{i=1}^n -\eps_i f(X_i)}
\\
& =
2 \E_{\eps} \sup_{f \in F}\pl{n\inv\sum_{i=1}^n \eps_i f(X_i)}
& \text{(symmetry of $\eps_i$)}
\\
& \leq 8 \frac{1}{\sqrt{n}} \sup_{j \in \N} \sqrt{\decr{\Tilde{p}_n}(j) \log(j+1)}
& \text{(reindexing)}
\\
& =
\frac{8}{\sqrt n}
\sqrt{S(\decr{\Tilde{p}_n}(j))}.
\end{align*}
Substituting into \eqref{eq:mcdiarmid-application} yields
\begin{align}
\label{eq:empirical-bound-1}
    \PR[X]{ \Delta_n(\mu)
    \le
    \frac{16}{\sqrt n}
    \sqrt{S(\decr{\Tilde{p}_n})} + \sqrt{\frac{8}{n} \log \frac{1}{\delta} } }
    & \geq
    1-\delta,
\end{align}
for all $\delta > 0$.
Taking $\E_{\Tilde{X}}[\cdot]$ on both sized completes the proof.
\qed

\section{Auxiliary results}
\label{sec:aux}

\subsection{
Worst-case optimality of
\eqref{eq:distr-free}
}

\begin{proposition}
\label{cor:vc}
There is an absolute constant $c>0$ such that
the following holds.
For
any
\(d,n \in \N\) with \( d \geq 4\)
and any
estimator
mapping $(x_1,\ldots,x_n)\in\paren{\set{0,1}^d}^n$
to $\tilde p_n\in [0,1]^d$,
there is a 
product
distribution $\mu$ on $\set{0,1}^d$
such that
\beqn
\label{eq:vc-lb}
\E\nrm{\tilde p_n-p}_\infty
&\ge&
c\paren{1\mn
\sqrt{\frac{\log d}{n}}
}
.
\eeqn
\end{proposition}

The proof relies on applying the Generalized Fano method \citep[Lemma~3]{yu1997assouad}:
\begin{lemma}[\citep{yu1997assouad}]
\label{lem:fano}
    For 
    \(r \geq 2\), let 
    \(\mathcal{M}_r 
    \)
    be a collection of $r$
    probability measures
    \(\nu_1, \nu_2, ... ,\nu_r\)
    with some parameter of interest
    \(\theta(\nu)\) 
    taking
    values in pseudo-metric space \( (\Theta, \rho) \)
    such that for all
    \(j \neq k \), we have
    \[
    \rho(\theta(\nu_j), \theta(\nu_{k}) )
    \geq
    \alpha
    \]
    and
    \[
    D(\nu_j ~\Vert~ \nu_{k})
    \leq
    \beta.
    \]
    Then
    \[
    \inf_{\hat\theta}
    \max_{j \in [d]}
    \E_{Z \sim \mu_j}
    \rho(\hat\theta(Z), \theta(\nu_j) )
    \geq
    \frac{\alpha}{2} \paren{1 - \paren{\frac{\beta + \log 2}{\log r}}},
    \]
    where the infimum is over all estimators 
    \(\hat\theta:Z\mapsto\Theta\).
\end{lemma}

\begin{proof}
Let \(\mu_1, \mu_2, ... ,\mu_d\) be
the
product measures on
\(\set{0,1}^d\) given by
\[\mu_i = \prod_{j=1}^d \Bernu\paren{\frac{1}{2} + \alpha\pred{i=j}}
, \qquad
i \in [d],\]
where
\(\alpha \in [0,1/4]\)
will be chosen later.
We will invoke Lemma~\ref{lem:fano} with \(r=d\), \(\nu_j = \mu_j^n\) for \(j \in [d]\),
\(\theta(\mu_j^n) = \E_{X\sim\mu_j} X\) and \(\rho = \nrm{\cdot}_\infty\).
We begin by verifying that the conditions of Lemma~\ref{lem:fano} apply.
Indeed,
for \(i \neq j\), \(i,j \in [d]\) we have 
\[
\rho(\theta(\mu_i^n), \theta(\mu_{j}^n) )
=
\nrm{\E_{X\sim\mu_i}X-\E_{Y\sim\mu_j}Y}_\infty \geq \alpha 
\]
and
\begin{align*}
D(\mu_i^n ~\Vert~ \mu_j^n)
&=
n D(\mu_i ~\Vert~ \mu_j)
&
\\
& =
n D\paren{\frac{1}{2} + \alpha ~\big\Vert~ \frac{1}{2}}
+
n D\paren{\frac{1}{2} ~\big\Vert~ \frac{1}{2} + \alpha}
&
\\
& \leq
n \paren{\frac{\alpha^2}{1-1/2} + \frac{\alpha^2}{1/2} + \frac{\alpha^2}{1-1/2 - \alpha} + \frac{\alpha^2}{1/2 + \alpha}}
&
\\
& =
n \alpha^2 \paren{\frac{8-16 \alpha ^2}{1-4 \alpha ^2}}
&
\\
& \leq
\frac{28}{3} n \alpha^2,
&
\end{align*}
where,
as in the proof of Lemma~\ref{lem:kl-upper-bound},
we used
\citet[Theorem~5]{gibbs02}.
Invoking Lemma~\ref{lem:fano}, 
\begin{align*}
    \sup_{\mu} \E\nrm{\tilde p_n-p}_\infty
    & \geq
    \max_{\mu_i, i\in[d]} \E\nrm{\tilde p_n-p}_\infty
    \\
    & \geq
    \frac{\alpha}{2} \paren{1 - \frac{\frac{28}{3} n \alpha^2 - \log 2}{\log d}}.
\end{align*}
We choose
\(
\alpha =\frac{1}{4} \mn \frac{\sqrt{\log (d/2)}}{2 \sqrt{7n} }
\)
and consider the two cases:
$\alpha<\frac14$
and
$\alpha=\frac14$.
If \(\alpha < \frac{1}{4}\),
then
\begin{align*}
    \sup_{\mu} \E\nrm{\tilde p_n-p}_\infty
    & \geq
    \frac{\alpha}{2} \paren{1 - \frac{\frac{28}{3} n \alpha^2 - \log 2}{\log d}}
    \\
    & =
    \frac{\sqrt{\log \frac{d}{2}} \left(1-\frac{\frac{1}{3} 
    \log \frac{d}{2}
    +\log 2}{\log d}\right)}{4 \sqrt{7n}}
    \\
    & = 
    \frac{\log ^{\frac{3}{2}}\left(\frac{d}{2}\right)}{6 \sqrt{7n} \log d}
    \\ 
    & \geq
    \frac{\sqrt{\log d}}{48 \sqrt{7 n}},
\end{align*}
where we used the fact \(\log \frac{d}{2} \geq \frac{\log d}{4}\) for \(d \geq 4\).
If \(\alpha = \frac{1}{4}\), then 
\(d \geq 2 \mathe^{\frac{7 n}{4}} \),
and hence
\begin{align*}
    \sup_{\mu} \E\nrm{\tilde p_n-p}_\infty
    & \geq
    \frac{\alpha}{2} \paren{1 - \frac{\frac{28}{3} n \alpha^2 - \log 2}{\log d}}
    \\
    & =
    \frac{1}{8} \left(1-\frac{\frac{7 n}{12}+\log 2}{\log d}\right)
    \\
    & \geq
    \frac{7 n}{84 n+48 \log 2}
    \\
    & \geq
    \frac{1}{17}.
    \end{align*}
It follows that
\[
\sup_{\mu} \E\nrm{\tilde p_n-p}_\infty
\geq
\frac{1}{17} \mn \frac{\sqrt{\log d}}{48 
\sqrt{7n}}
\]
holds for both cases.
\end{proof}

\subsection{Maximal inequalities
}
\begin{lemma}
[Maximal inequality for inhomogeneous sub-Gaussians]
\label{lemma:max-subgaussian}
Let $Y_1,Y_2,...$ be random variables and $\sigma_1,\sigma_2,...$ positive real numbers such that
\beq
\PR{Y_i \geq  t} \leq e^{-t^2/2\sigma_i^2},
\qquad i\in\N,~t\ge0.
\eeq
Let
\beq
T &\eqdef&
\sup_{i \in \N}
\sigma_i^2 \log (i+1)
.
\eeq
Then
\beq
\E\sup_{i \in \N} \pl{Y_i}
&\leq&
4\sqrt{T}.
\eeq
\end{lemma}

\begin{proof}
By the union bound, for $t^2 > 2T $ we have
\begin{align*}
    \PR{\sup_{i \in \N} \pl{Y_i} \geq t} 
    &\leq \sum_{i=1}^\infty
    \PR{\pl{Y_i} \geq  t}
    \\
    & = \sum_{i=1}^\infty
    \PR{Y_i \geq  t}
    \\
    &\leq \sum_{i=1}^\infty
    \mathe^{-t^2/2\sigma_i^2}
    \\
    &\leq \sum_{i=1}^\infty
    \mathe^{-t^2\log(i+1)/2T}
    \\
    &= \sum_{i=2}^\infty
    i^{-t^2/2T}
    \\
    &\leq \int_{1}^\infty
    u^{-t^2/2T} \mathd u
    \\
    &= \frac{2T}{t^2 - 2T}.
\end{align*}
Integrating,
\begin{align*}
    \E\sup_{i \in \N} \pl{Y_i} & \leq 
    \int_0^\infty
    \PR{\sup_{i \in \N} \pl{Y_i} \geq t}\mathd t
    \\
    & \leq 2\sqrt{T} + \int_{2\sqrt{T}}^\infty \frac{2T}{t^2 - 2T} \mathd t
    \\
    & = 2\sqrt{T} + \sqrt{T}\frac{-\log \left(3-2 \sqrt{2}\right)}{\sqrt{2}}
    \\ & \leq 4\sqrt{T}.
\end{align*}
\end{proof}

\begin{lemma}
[Maximal inequality for inhomogeneous sub-gammas]
\label{lem:max-subgamma}
    Let
    $Y_{i\in I \subseteq \N}$
    random variables
    such that, for each $i \in I$, there 
    are $v_i>0$ and $a_i \geq 0$
satisfying either of the conditions
\begin{enumerate}
\item[\textup{(a)}] 
For all $0<t<\frac{1}{a_i}$ (or all $0<t$ if $a_i=0$),
        \beq
        \E[\exp(t Y_i)]
        &\le&
        \exp\paren{
        \frac{v_i t^2}{
        2[1-a_i t]
        }
        }.
        \eeq
\item[\textup{(b)}]
        For all $ \eps \geq 0$,
        \beq
            \PR{Y_i \geq \eps}
            & \leq&
            \exp\paren {- \frac{\eps^2}{2(v_i + a_i \eps)}}.
        \eeq
    \end{enumerate}
    Then
\beqn
\label{eq:max-subgamma}
    \E{\sup_{i \in I} \pl{Y_i}  } 
    &\leq&
    12 \sup_{i \in I}
    \sqrt{ v_i \log (i+1) }
    +
    16\sup_{i \in I} a_i \log (i+1)
    .
\eeqn
\end{lemma}
\noindent
{\bf Remark.}
It is instructive to compare
this to
\citet[Corollary 2.6]{blm13},
which estimates
$\E\sup_{i\in I}Y_i\lesssim
\sqrt{v\log d}+a\log d
$
in the 
finite, homogeneous
special case
$v_i\equiv v$,
$a_i\equiv a$,
and
$|I|=d$.
\begin{proof}
To streamline the proof
we only consider the case where 
$a_i>0$ for all $i \in I$;
the argument is analogous
if some of them are zero.
By Cram\'er-Chernoff's method,
any $Y_i$
satisfying (a) also satisfies (b):
\beq
\PR{Y_i \geq \eps}
& \leq& \inf_{0<t<\frac{1}{a}}
            \frac{\exp\paren{
            \frac{v_i t^2}{
            2[1-a_i t]
            }}}{\mathe^{t \eps}}
            \\
            & \leq&
            \exp\paren {- \frac{\eps^2}{2(v_i + a_i \eps)}}.
\eeq
Hence,
for each $i \in I$ and all $ \delta > 0$,
\beqn
            \PR{\pl{Y_i} \geq \sqrt{2 v_i \log \frac{1}{\delta} }
            +
            2 a_i \log \frac{1}{\delta}}
            & \leq&
            \PR{Y_i \geq \sqrt{2 v_i \log \frac{1}{\delta} 
            +
            \paren{a_i \log \frac{1}{\delta}}^2}
            +
            a_i \log \frac{1}{\delta}}
            \nonumber
            \\
            & \leq&
            \delta
\label{eq:subgamma-concentration}
\eeqn
where we used the subadditivity of $\sqrt{\cdot}$.
Applying the union bound to the family of inequalities
in
\eqref{eq:subgamma-concentration},
        \begin{align*}
           &\PR{\sup_{i \in I} \pl{Y_i} 
            \geq \sup_{i \in I} \sqrt{2 v_i \log \frac{i(i+1)}{\delta} } + 2 a_i \log \frac{i(i+1)}{\delta}
        }
        \\
        \leq
        \sum_{i \in I}
        &\PR{\pl{Y_i} 
            \geq \sup_{i \in I} \sqrt{2 v_i \log \frac{i(i+1)}{\delta} } + 2 a_i \log \frac{i(i+1)}{\delta}
        }
        \\
        \leq
        \sum_{i \in I}
        &\PR{\pl{Y_i} 
            \geq \sqrt{2 v_i \log \frac{i(i+1)}{\delta} } + 2 a_i \log \frac{i(i+1)}{\delta}
        }
        \\
        \leq
        \sum_{i \in I} & \frac{\delta}{i(i+1)}
        \leq
        \sum_{i \in I}  \frac{\delta}{i(i+1)}
        =
        \delta.
        \end{align*}
        Let $Y \eqdef \sup_{i \in I} \pl{Y_i},\ a_i^* \eqdef \sup_{i \in I} a_i,\ v_i^* \eqdef \sup_{i \in I} v_i^*$ and note that the above bound implies
\beq
\PR{Y 
    -
    \sup_{i \in I}
    \paren{ \sqrt{2 v_i \log i(i+1) } + 2 a_i \log i(i+1)
    }
    \geq
    \max \paren{
    \sqrt{8 v_i^* \log \frac{1}{\delta} } ,
    4 a_i^* \log \frac{1}{\delta}
    }
}
&\leq& \delta.
\eeq
Let $ Z \eqdef Y 
    -
    \sup_{i \in I}
    \paren{ \sqrt{2 v_i \log i(i+1) } + 2 a_i \log i(i+1)
    }$.
By a change of variable,
\beq
    \PR{Z 
    \geq
    \varepsilon
}
& \leq&
\exp{\paren{-\frac{\varepsilon^2}{8 v_i^*}}}
\mx
\exp{\paren{-\frac{\varepsilon}{4 a_i^*}}}
\\ & \leq &
\exp{\paren{-\frac{\varepsilon^2}{8 v_i^*}}}
+
\exp{\paren{-\frac{\varepsilon}{4 a_i^*}}}
.
\eeq
Integrating,
\beq
\E{Z} &=& \int_0^\infty \PR{Z \geq \varepsilon} \mathd\varepsilon
    \\
    & \leq& \int_0^\infty \exp{\paren{-\frac{\varepsilon^2}{8 v_i^*}}}
+
\exp{\paren{-\frac{\varepsilon}{4 a_i^*}}} \mathd\varepsilon
\\ &=&
\sqrt{2 \pi v_i^*}
+
4 a_i^*;
\eeq
this proves
\eqref{eq:max-subgamma}.
\end{proof}

\end{document}